\documentclass[12pt,reqno]{amsart}
\usepackage{amsfonts}
\usepackage{amsmath}
\usepackage{footnote}
\usepackage{enumerate}
\newtheorem{theorem}{Theorem}[section]

\title[On the Distributions of Product and Quotient..]{On the Distributions of Product and Quotient of two Independent $\hat{I}$-function variates}
\author[Vilma D'Souza, Shantha Kumari. K. and Arjun K. Rathie]{Vilma D'Souza$^{1}$, Shantha Kumari Kurumujji$^{1^*}$, Arjun K. Rathie$^{2}$}
\address{$^1$ Department of Mathematics, A J Institute of Engineering and Technology, Mangaluru-575006(Affiliated to Visvesvaraya Technological University (VTU), Belagavi),   Karnataka, INDIA.}
\email{dsouzavilma12@gmail.com, shanthakk99@gmail.com}
\address{$^2$ Department of Mathematics, Vedant College of Engineering \& Technology (Rajasthan Technical    University), Village: Tulsi, Post : Jakhamund, Dist. Bundi, Rajasthan State, India}
\email{arjunkumarrathie@gmail.com}
\thanks{$*$ Corresponding Author}
\begin{document}
\maketitle
\thispagestyle{plain}
\begin{abstract}
	The study of probability distributions for random variables and their algebraic combinations has been a central focus driving the advancement of probability and statistics.  Since the 1920s, the challenge of calculating the probability distributions of sums, differences, products, and quotients of independent random variables have drawn the attention of numerous statisticians and mathematicians who studied the algebraic properties and relationships of random variables.  Statistical distributions are highly helpful in data science and machine learning, as they provide a range of possible values for the variables,  aiding in the development of a deeper understanding of the underlying problem. In this paper, we have presented a new probability distribution based on the $\hat{I}$-function. Also, we have discussed the applications of the $\hat{I}$ function, particularly in deriving  the distributions of product and the quotient involving two independent $\hat{I}$ function variates. 
Additionally, it has been shown that both the product and quotient of two independent $\hat{I}$-function variates also follow the $\hat{I}$-function distribution. Furthermore, the new distribution, known as the $\hat{I}$-function distribution, includes several well-known classical distributions such as the gamma, beta, exponential, normal H-function, and G-function distributions, among others, as special cases.   Therefore, the $\hat{I}$-function distribution can be considered a characterization or generalization of the above-mentioned distributions. 
\end{abstract} 
\subjclass{33C60, 44A20, 60E05}  
\keywords{I function, Distributions, Mellin Transform, Random variable}        
\section{Introduction}
The study of special functions and their distributions is a fundamental area of  mathematical analysis, with numerous applications in statistics, physics, and engineering. Among these special functions, the H-function has been extensively researched for its rich algebraic and analytical properties. The H-function is interconnected with several other important mathematical functions. It is related to the generalized hypergeometric function ${}_pF_q$, the Meijer G-function, the generalized Wright function ${}_p\psi_q$, and the Meiger's G-function\cite{luke}. These relationships are thoroughly documented in the works of Mathai et al. \cite{mathai, mathai2010}, providing a comprehensive framework for understanding the H-function and its various applications.  Subsequent research into the H-function distribution was carried out by Cook \cite{cook}, who examined its algebraic properties, and later by Bodenschatz \cite{boden}, who further explored its statistical characteristics. Their studies provided  a more comprehensive understanding of  the mathematical structure and potential applications of the H-function. 

 Inayat-Hussain \cite{inayat} introduced a significant generalization of the H-function, known as the $\bar{H}$-function.  Further generalization beyond the $\bar{H}$-function has been provided by Rathie\cite{rathie1997}, who extended the framework to include a broader class of special functions, including the $\bar{H}$-function,  generalized Riemann zeta function, the polylogarithmic function of complex order, and the exact partition function of the Gaussian free energy model in statistical mechanics, which are not special cases of the H-function.    The extension of the I-function to two variables and its
 properties were studied by shantha et al.\cite{shanthalemat}.  
 
 Recent advancements have focused on novel methods for generating Fox's H-function distributed random variables, essential for modeling fading in wireless communication. These methods employ standard normal and Gamma random variables, validated via Monte Carlo simulations to closely approximate analytical H-function probability density functions (PDFs) across various parameters\cite{yousuf2019}. This is a topic of considerable interest due to its applications in various domains, including fading channels within wireless communication fields. The applications of I-function to wireless communication were also studied by Ansari et al.\cite{ansari2013, ansari2017}.  Motivated by this, our paper focuses on obtaining the distributions of the product and quotient of two independent random variables having densities in terms of the $\hat{I}$ - functions (which is a particular case of the Rathie's ${I}$ - function).  For another form of distribution, readers are suggested to refer to \cite{kataria}.

In 1997, Rathie introduced the I-function\cite{rathie1997} as a generalization of the H-function defined by Fox\cite{fox} and studied by Braaksma\cite{braaksma}. It is represented by the following Mellin-Barnes type contour integral: 
\begin{align} \label{I}
\qquad \qquad \mathrm{I_{p,q}^{m,n}}(z) & \equiv \mathrm{I_{p,q}^{m,n}}\left[ \begin{array}{c} z \end{array} \left|\begin{array}{l} \left(a_1, e_1, A_1),\ldots, (a_p, e_p, A_p \right)\\ \left(b_1,f_1, B_1), \ldots, (b_q, f_q, B_q\right) \end{array}  \right.\right]   \nonumber \\
& = \frac{1}{2\pi i} \int_{\mathcal{L}} \theta(s) z^s ds 
\end{align}
where 
\begin{equation} \label{Iphi}
\qquad \qquad \theta(s) = \frac{\prod_{j=1}^m \Gamma^{B_j} \left(b_j-f_js\right) \prod_{j=1}^n\Gamma^{A_j} \left(1-a_j+e_js\right)}{\prod_{j=m+1}^q \Gamma^{B_j} \left(1-b_j+f_js\right) \prod_{j=n+1}^p\Gamma^{A_j} \left(a_j-e_js\right)}
\end{equation}
Also 
\begin{enumerate}[(i)]
\item $i$ is the imaginary unit, $i=\sqrt{-1}$;
\item $z$ is a non zero  complex variable, signifying the argument of the I-function;  ;
\item $ m,n,p,q$ are integers satisfying the constraints,  $0 \leq m \leq q$,  $0\leq n \leq p$; 
\item  The contour $\mathcal{L}$ is a carefully chosen path in the complex plane for the integration;
\item an empty product is to be interpreted as unity; 
\item The parameters $e_j$ and $f_j$, with $j$ ranging from $1$ to $p$ and $q$ respectively, represent positive numbers;  
\item The parameters $A_j$ and $B_j$, with $j$ ranging from $1$ to $p$ and $q$ respectively, also denote positive numbers; 
\item The parameters $a_j$ and $b_j$, with $j$ ranging from $1$ to $p$ and $q$, respectively, are complex numbers, subject to the condition that no singularity of $\Gamma^{B_j} \left(b_j-f_js\right)$, for $j=1,\ldots,m$, coincides with any singularity of $\Gamma^{A_j} \left(1-a_j+e_js\right)$, for $j=1,\ldots,n$. It is important to note that these singularities are not necessarily poles.
\item  Additionally, the contour $\mathcal{L}$ is described to go from $\sigma - i \infty$ to $\sigma + i\infty$ ($\sigma$ real), with the requirement that the singularities of $\Gamma^{B_j} \left(b_j-f_js\right)$ for $j=1,\ldots,m$ lie to the right of $\mathcal{L}$, and the singularities of $\Gamma^{A_j}\left(1-a_j+e_js\right)$ for $j=1,\ldots,n$ lie to the left of $\mathcal{L}$. 
 \end{enumerate}
In a compact notation, \eqref{I} is represented as,
\begin{equation*} 
  \mathrm{I_{\:p,\;q}^{\:m,\;n}} \left[\begin{array} {c} z \end{array}  \left| \begin{array}{l} _1(a_j,e_j,A_j)_p\\ _1(b_j,f_j,B_j)_q  \end{array} \right. \right] 
\end{equation*}
The function defined by \eqref{I} is convergent if 
\begin{equation} \label{Icond}
  \Delta>0, \quad |arg (z)|<\frac{1}{2}\Delta\pi,
\end{equation}
 where
\begin{equation} \label{Idelta}
  \Delta = \sum_{j=1}^m B_jf_j - \sum_{j=m+1}^q B_jf_j + \sum_{j=1}^n A_je_j - \sum_{j=n+1}^p A_je_j.
\end{equation}
When $A_1= A_2 =\cdots = A_p= 1= B_1= B_2= \cdots = B_q$, \eqref{I} reduces to the H-function introduced by Fox\cite{fox} and studied by Braaksma\cite{braaksma}.\\

When the exponents $A_j=1$ for $j=1,2,\cdots,n$ and $B_j=1$ for $j=1,2,\cdots,m$, the $I$-function defined by \eqref{I} reduces to $\hat{I}$-function, which is represented by the following Mellin Barnes  type contour integral:
 \begin{align}
 \label{hatI}
 \hat{I}(z) & = \hat{I}_{\;p,\; q}^{\;m,\; n} \left[ z \left |\begin{array}{c} _1(a_j, e_j, 1)_n,\; _{n+1}(a_j, e_j,A_j)_p\\ _1(b_j, f_j, 1)_m,\; _{m+1}(b_j, f_j, B_j)_q\end{array} \right] \right. \nonumber\\
 \\ \nonumber
 & = \frac{1}{2\pi i} \int_{\mathcal{L}} \theta(s) z^s ds 
\end{align}
where 
\begin{equation} \label{hatItheta}
  \theta(s) = \frac{\prod_{j=1}^m \Gamma \left(b_j-f_js\right) \prod_{j=1}^n\Gamma \left(1-a_j+e_js\right)}{\prod_{j=m+1}^q \Gamma^{B_j} \left(1-b_j+f_js\right) \prod_{j=n+1}^p\Gamma^{A_j} \left(a_j-e_js\right)}
\end{equation}
and the function defined by \eqref{hatItheta} is convergent if 
\begin{equation} 
 \Delta_1>0, \quad |arg (z)|<\frac{1}{2}\Delta_1\pi
\end{equation}
 where
\begin{equation} \label{hatIdelta}
  \Delta_1 = \sum_{j=1}^m f_j - \sum_{j=m+1}^q B_jf_j + \sum_{j=1}^n e_j - \sum_{j=n+1}^p A_je_j.
\end{equation}
Also, from Rathie \cite{rathie1997}, it can be proven that \\
$\hat{I}(z)\sim z^c$ for small $z$, where $\; c= \displaystyle \min_{1\leq j \leq m}\left(Re\left[\frac{b_j}{f_j}\right]\right).$\\
and \\
$\hat{I}(z)\sim z^d$ for large $z$, where $\; d=\displaystyle \max_{1\leq j \leq n}\left(Re\left[\frac{a_j-1}{e_j}\right]\right).$
 
By the definition of $\hat{I}$-function, the following properties are immediate. 
\begin{align} \label{prop1}
\hat{I}_{\;p,\; q}^{\;m,\; n} & \left[ z^\sigma \left |\begin{array}{c} _1(a_j, e_j, 1)_n,\; _{n+1}(a_j, e_j,A_j)_p\\ _1(b_j, f_j, 1)_m,\; _{m+1}(b_j, f_j, B_j)_q\end{array} \right] \right. \nonumber \\
& = \frac{1}{\sigma} \hat{I}_{\;p,\; q}^{\;m,\; n} \left[ z \left |\begin{array}{c} _1(a_j, \frac{e_j}{\sigma}, 1)_n,\; _{n+1}(a_j, \frac{e_j}{\sigma}, A_j)_p\\ _1(b_j, \frac{f_j}{\sigma}, 1)_m,\; _{m+1}(b_j, \frac{f_j}{\sigma}, B_j)_q\end{array} \right] \right. , ~~~ \sigma>0
\end{align}
\begin{align} \label{prop2}
z^{\sigma} \; \hat{I}_{\;p,\; q}^{\;m,\; n} & \left[ z \left |\begin{array}{c} _1(a_j, e_j, 1)_n,\; _{n+1}(a_j, e_j,A_j)_p\\ _1(b_j, f_j, 1)_m,\; _{m+1}(b_j, f_j, B_j)_q\end{array} \right] \right. \nonumber \\
& =  \hat{I}_{\;p,\; q}^{\;m,\; n} \left[ z \left |\begin{array}{c} _1(a_j+\sigma e_j, e_j, 1)_n,\; _{n+1}(a_j+ \sigma e_j, e_j, A_j)_p\\ _1(b_j+ \sigma f_j, f_j, 1)_m,\; _{m+1}(b_j+ \sigma f_j, f_j, B_j)_q\end{array} \right] \right. 
\end{align}

In addition to this, the Mellin transform of a continuous positive random variable $X$ with $\operatorname{pdf} f_X(x)$ is defined by
\begin{align}\label{mt}
\mathcal{M}_{f_X}(s)=E\left[X^{s-1}\right]=\int_0^{\infty} x^{s-1} f_X(x) d x
\end{align}
There exists an inversion integral
\begin{align} \label{imt}
f_X(x)=\frac{1}{2 \pi i} \int _{\mathcal{L}} x^{-s} \mathcal{M}_{f_X}(s) ds.
\end{align}
where $\mathcal{L}$ is a suitable contour.
\section{Mellin Transform of  $\hat{I}$-function}
In this section, we shall establish the Mellin transform of $\hat{I}$-function.\\
By the definition of $\hat{I}$-function \eqref{hatI}, we have
\begin{align} 
 \hat{I}(z\; x^\sigma) & = \hat{I}_{\;p,\; q}^{\;m,\; n} \left[ z\;x^\sigma \left |\begin{array}{c} _1(a_j, e_j, 1)_n,\; _{n+1}(a_j, e_j,A_j)_p\\ _1(b_j, f_j, 1)_m,\; _{m+1}(b_j, f_j, B_j)_q\end{array} \right] \right. \nonumber\\
 & = \frac{1}{2\pi i} \int_{\mathcal{L}} \theta(s) (z\;x^\sigma)^{-s} ds 
 \end{align}
Replacing $\sigma s$ by $t$ and $ds$ by $\frac{dt}{\sigma}$, we get
\begin{align} 
\hat{I}(z\; x^\sigma) =  \frac{1}{2\pi i} \int_{\mathcal{L}} x^{-t} \; \theta\left (\frac{t}{\sigma}\right ) \;  z^{\frac{-t}{\sigma}} \;  \frac{dt}{\sigma} 
 \end{align}
By the definition of Inverse Mellin Transform \eqref{imt}, we have
\begin{align}\label{hatImt}
\mathcal{M} \{\hat{I}(z\; x^\sigma)\} = \frac {\theta(\frac{s}{\sigma})}{\sigma \; z^\frac{s}{\sigma}}
 \end{align}
In other way, the result \eqref{hatImt} can be written in the following form:
\begin{align} 
 \int_0^{\infty} x^{s-1} \;&  \hat{I}_{\;p,\; q}^{\;m,\; n} \left[ z\;x^\sigma \left |\begin{array}{c} _1(a_j, e_j, 1)_n,\; _{n+1}(a_j, e_j,A_j)_p\\ _1(b_j, f_j, 1)_m,\; _{m+1}(b_j, f_j, B_j)_q\end{array} \right] \right. dx \nonumber\\
 & = \frac{1}{\sigma \; z^{\frac{s}{\sigma}}}\;  \frac{\prod_{j=1}^m \Gamma \left(b_j-f_j \frac{s}{\sigma}\right) \prod_{j=1}^n\Gamma \left(1-a_j+e_j\frac{s}{\sigma}\right)}{\prod_{j=m+1}^q \Gamma^{B_j} \left(1-b_j+f_j\frac{s}{\sigma}\right) \prod_{j=n+1}^p\Gamma^{A_j} \left(a_j-e_j\frac{s}{\sigma}\right)}  
 \end{align}
 where 
 \begin{align}
 Re\left [s+ \sigma \min_{1\leq j \leq m}\left(\frac{b_j}{f_j}\right)\right]> 0, \nonumber \\
 Re\left [s+ \sigma \min_{1\leq j \leq n}\left(\frac{a_j-1}{e_j}\right)\right]> 0, \nonumber \\
 \Delta_1>0, ~~~~\; |arg z| < \frac{\Delta_1 \pi}{2}, \nonumber
 \end{align}
  where $\Delta_1$ is the same as defined by \eqref{hatIdelta}. 
\section{Mellin Transform of the product of two $\hat{I}$-functions}
In this section, we shall establish the Mellin transform of the product of two $\hat{I}$-functions.\\
\\
For this, let 
\begin{align}
 \hat{I}(z_1x^\sigma) & = \hat{I}_{\;p_1,\; q_1}^{\;m_1,\; n_1} \left[ z_1x^\sigma \left |\begin{array}{c} _1(a_j, e_j, 1)_{n_1},\; _{n_1+1}(a_j, e_j,A_j)_{p_1}\\ _1(b_j, f_j, 1)_{m_1},\; _{m_1+1}(b_j, f_j, B_j)_{q_1}\end{array} \right] \right. \nonumber\\
 & = \frac{1}{2\pi i} \int_{\mathcal{L}} \theta_1(s_1) z_1^{-s_1} ds_1 
\end{align}
and 
\begin{align}
 \hat{I}(z_2x^\mu) & = \hat{I}_{\;p_2,\; q_2}^{\;m_2,\; n_2} \left[ z_2x^\mu \left |\begin{array}{c} _1(a_j', e_j', 1)_{n_2},\; _{n_2+1}(a_j', e_j',A_j')_{p_2}\\ _1(b_j', f_j', 1)_{m_2},\; _{m_1+1}(b_j', f_j', B_j')_{q_2}\end{array} \right] \right. \nonumber\\
 & = \frac{1}{2\pi i} \int_{\mathcal{L}} \theta_2(s_2) z_2^{-s_2} ds_2 
\end{align}
where 
\begin{equation} \label{theta1s1}
\theta_1(s_1) = \frac{\prod_{j=1}^{n_1} \Gamma \left(1-a_j-e_j s_1\right) \prod_{j=1}^{m_1} \Gamma\left(b_j+f_j s_1\right)}{\prod_{j=n_1+1}^{p_1} \Gamma^{A_j} \left(a_j+e_j s_1\right) \prod_{j=m_1+1}^{q_1}\Gamma^{B_j} \left(1-b_j-f_j s_1\right)}
\end{equation}
and 
\begin{align} \label{theta2s2}
\theta_2(s_2)= \frac{\prod_{j=1}^{n_2} \Gamma \left(1-a_j'-e_j' s_2\right) \prod_{j=1}^{m_2} \Gamma \left(b_j'+f_j' s_2\right)}{\prod_{j=n_2+1}^{p_2} \Gamma^{A_j'} \left(a_j'+e_j' s_2\right) \prod_{j=m_2+1}^{q_2}\Gamma^{B_j'} \left(1-b_j'-f_j 's_2\right)}
\end{align}
Then by the definition of Mellin transform,
\begin{align}
M_{\hat{I}[z_1x^\sigma] \; \hat{I}[z_2x^\mu]} (s) & = \int_{0}^{\infty}  x^{s-1} \hat{I}[z_1x^\sigma] \; \hat{I}[z_2x^\mu] \; dx \nonumber \\
& = \int_{0}^{\infty}\left\{  {x^{s-1}}{\hat{I}_{\:p_1,\;q_1}^{\:m_1,\;n_1}} \left[\begin{array} {c} z_1 \;x^{\sigma }\end{array}  \left| \begin{array}{c} _1(a_j, e_j, 1)_{n_1}, \;  _{n_1+1}(a_j, e_j, A_j)_{p_1}\\ _1(b_j, f_j, 1)_{m_1}, \; _{m_1+1}(b_j, f_j, B_j)_{q_1} \end{array} \right. \right] \right. \nonumber \\
& \qquad \left. \times {\hat{I}_{\:p_2,\;q_2}^{\:m_2,\;n_2}} \left[\begin{array} {c} z_2 \;x^{\mu} \end{array}  \left| \begin{array}{c}  _1(a_j', e_j', 1)_{n_2}, \; _{n_2+1}(a_j', e_j', A_j')_{p_2}\\ _1(b_j', f_j', 1)_{m_2}, \; _{m_2+1}(b_j', f_j', B_j')_{q_2} \end{array} \right. \right] \right\} dx \nonumber 
\end{align}
Expressing one of the $\hat{I}$-functions by its definition, we have
\begin{align}
 M_{\hat{I}[z_1x^\sigma] \; \hat{I}[z_2x^\mu]} (s)  = \int_{0}^{\infty}  x^{s-1} \hat{I}[z_1x^\sigma] \;  \frac{1}{2\pi i} \int_{\mathcal{L}} \theta_2(s_2)\; (z_2x^\mu)^{-s_2}\ ds_2\;  dx \nonumber 
\end{align}
By interchanging the order of integration, we have
\begin{align}
M_{\hat{I}[z_1x^\sigma] \; \hat{I}[z_2x^\mu]} (s) &= \int_{\mathcal{L}} \theta_2(s_2)\; z_2^{-s_2}  \int_{0}^{\infty}  x^{s-\mu s_2-1} \hat{I}[z_1x^\sigma] \; dx \; ds_2 \nonumber 
\end{align}
Using \eqref{mt},
\begin{align}
M_{\hat{I}[z_1x^\sigma] \; \hat{I}[z_2x^\mu]} (s)= \int_{\mathcal{L}} \theta_2(s_2)\; z_2^{-s_2} \; M_{\hat{I}[z_1x^\sigma] } (s-\mu s_2) \; ds_2 \nonumber 
\end{align}
Using \eqref{hatImt},
\begin{align}
&M_{\hat{I}[z_1x^\sigma] \; \hat{I}[z_2x^\mu]} (s) = \int_{\mathcal{L}} \theta_2(s_2)\; z_2^{-s_2} \;  \frac{1}{\sigma z_1^{\frac{s}{\sigma}- \frac{\mu s_2}{\sigma}}} \; \theta_1\left(\frac{s-\mu s_2}{\sigma}\right) \; ds_2\nonumber
\end{align}
Using \eqref{hatI},we have finally
\begin{align} \label{mtprod}
& M_{\hat{I}[z_1x^\sigma] \; \hat{I}[z_2x^\mu]} (s)  \nonumber \\
&= \frac{1}{\sigma{ z_1}^{\frac{s}{\sigma}}} \; {I_{\:q_1+p_2, \;p_1+q_2}^{\:n_1+m_2,\;m_1+n_2}} \left[\begin{array} {c} \frac{z_2} { z_1^\frac {-\mu}{\sigma}}  \end{array}  \left| \begin{array}{c} _1(a_j', e_j', 1)_{n_2}, ~~{}_1(1-b_j-\frac{s}{\sigma} f_j, \frac{\mu}{\sigma} f_j, 1)_{m_1}, \\ _1(b_j', f_j', 1)_{m_2}, \; _1(1-a_j-\frac{s}{\sigma} e_j, \frac{\mu}{\sigma} e_j, 1)_{n_1} \end{array} \right. \right.  \nonumber \\
& \qquad \qquad \qquad  \left. \left. \begin{array} {c} _{m_1+ 1}(1-b_j-\frac{s}{\sigma} f_j, \frac{\mu}{\sigma} f_j, B_j)_{q_1}, \; _{n_2+1}(a_j', e_j', A_j')_{p_2}\\  _{n_1+1}(1-a_j-\frac{s}{\sigma} e_j, \frac{\mu}{\sigma} e_j, A_j)_{p_1},  \;_{{m_2}+1}(b_j', f_j', B_j')_{q_2}  \end{array} \right. \right] 
\end{align}
In either way, the result \eqref{mtprod} can be written in the following form:
\begin{align} 
 \int_0^{\infty} & x^{s-1} \;  \hat{I}_{\;p_1,\; q_1}^{\;m_1,\; n_1} \left[ z_1x^\sigma \left |\begin{array}{c} _1(a_j, e_j, 1)_{n_1},\; _{n_1+1}(a_j, e_j,A_j)_{p_1}\\ _1(b_j, f_j, 1)_{m_1},\; _{m_1+1}(b_j, f_j, B_j)_{q_1}\end{array} \right] \right.  \nonumber\\
  & \times  \hat{I}_{\;p_2,\; q_2}^{\;m_2,\; n_2} \left[ z_2x^\mu \left |\begin{array}{c} _1(a_j', e_j', 1)_{n_2},\; _{n_2+1}(a_j', e_j',A_j')_{p_2}\\ _1(b_j', f_j', 1)_{m_2},\; _{m_1+1}(b_j', f_j', B_j')_{q_2}\end{array} \right] \right. \; dx\nonumber \\
 & = \frac{1}{\sigma{ z_1}^{\frac{s}{\sigma}}} \; {I_{\:q_1+p_2, \;p_1+q_2}^{\:n_1+m_2,\;m_1+n_2}} \left[\begin{array} {c} z_2  z_1^\frac {\mu}{\sigma}  \end{array}  \left| \begin{array}{c} _1(a_j', e_j', 1)_{n_2}, \; _1(1-b_j-\frac{s}{\sigma} f_j, \frac{\mu}{\sigma} f_j, 1)_{m_1}, \\ _1(b_j', f_j', 1)_{m_2}, \; _1(1-a_j-\frac{s}{\sigma} e_j, \frac{\mu}{\sigma} e_j, 1)_{n_1} \end{array} \right. \right.  \nonumber \\
& \qquad \qquad \qquad\left. \left. \begin{array} {c} _{m_1+ 1}(1-b_j-\frac{s}{\sigma} f_j, \frac{\mu}{\sigma} f_j, B_j)_{q_1}, \; _{n_2+1}(a_j', e_j', A_j')_{p_2}\\  _{n_1+1}(1-a_j-\frac{s}{\sigma} e_j, \frac{\mu}{\sigma} e_j, A_j)_{p_1},  \;_{{m_2}+1}(b_j', f_j', B_j')_{q_2}  \end{array} \right. \right]  
 \end{align}
 where 
 \begin{align}
 Re\left [s+ \sigma \min_{1\leq j \leq m_1}\left(\frac{b_j}{f_j}\right)+\mu \min_{1\leq j \leq m_2}\left(\frac{b_j'}{f_j'}\right)\right]> 0, \nonumber \\
  Re\left [s+ \sigma \min_{1\leq j \leq n_1}\left(\frac{a_j-1}{e_j}\right)+ \mu \min_{1\leq j \leq n_2}\left(\frac{a_j'-1}{e_j'}\right)\right]> 0, \nonumber \\
  \Delta_{11}>0, \; |arg z_1| < \frac{ \Delta_{11} \pi}{2},\; \Delta_{12}>0, \; |arg z_2| < \frac{ \Delta_{12} \pi}{2}, \nonumber  
  \end {align}
where $ \Delta_{11}$ and $ \Delta_{12}$ are defined by 
\begin{equation} 
   \Delta_{11} = \sum_{j=1}^{m_1} f_j - \sum_{j=m_1+1}^{q_1} B_jf_j + \sum_{j=1}^{n_1} e_j - \sum_{j=n_1+1}^{p_1} A_je_j\nonumber
\end{equation}
and 
\begin{equation} 
   \Delta_{12} = \sum_{j=1}^{m_2} f_j' - \sum_{j=m_2+1}^{q_2} B_j'f_j' + \sum_{j=1}^{n_2} e_j' - \sum_{j=n_2+1}^{p_2} A_j'e_j' \nonumber
\end{equation}
\section{Distributions of product and quotient of two independent $\hat{I}$-function variates}

In this section, we shall obtain the distribution of product and quotient of two independent $\hat{I}$-function variates asserted in the following theorem. 
\begin{theorem} If $X_1$ and $X_2$ be the independent $\hat{I}$-function variates with p.d.f. given respectively by 
\begin{align} \label{f1x1}
f_1(X_1) = \sigma \; z_1^{\frac{s_1}{\sigma}} \; \phi_1(s_1) \; x^{s_1-1}\hat{I}_{\;p_1,\; q_1}^{\;m_1,\; n_1} \left[ z_1\; x^\sigma \left |\begin{array}{c} _1(a_j, e_j, 1)_{n_1},\; _{n_1+1}(a_j, e_j,A_j)_{p_1}\\ _1(b_j, f_j, 1)_{m_1},\; _{m_1+1}(b_j, f_j, B_j)_{q_1}\end{array} \right] \right.
\end{align}
and 
\begin{align} \label{f2x2}
f_2(X_2) = \mu \; z_2^{\frac{s_2}{\mu}} \; \phi_2(s_2) \; x^{s_2-1}\hat{I}_{\;p_2,\; q_2}^{\;m_2,\; n_2} \left[ z_2\; x^\mu \left |\begin{array}{c} _1(a_j', e_j', 1)_{n_2},\; _{n_2+1}(a_j', e_j',A_j')_{p_2}\\ _1(b_j', f_j', 1)_{m_2},\; _{m_2+1}(b_j', f_j', B_j')_{q_2}\end{array} \right] \right.
\end{align}
where $0<\sigma <\infty$, $0<\mu <\infty$, $z_1$, $z_2$, $s_1$, $s_2 >0$, with\\
\begin{align}\label{phi1s1} 
\phi_1(s_1) = \frac{\prod_{j=n_1+1}^{p_1} \Gamma^{A_j} \left(a_j+\frac{s_1}{\sigma}e_j \right) \prod_{j=m_1+1}^{q_1}\Gamma^{B_j} \left(1-b_j-\frac{s_1}{\sigma}f_j \right)}{\prod_{j=1}^{n_1} \Gamma \left(1-a_j-\frac{s_1}{\sigma}e_j \right) \prod_{j=1}^{m_1} \Gamma\left(b_j+\frac{s_1}{\sigma}f_j \right)}
\end{align} 
and
\begin{align} \label{phi2s2}
\phi_2(s_2) = \frac{\prod_{j=n_2+1}^{p_2} \Gamma^{A_j'} \left(a_j'+\frac{s_2}{\mu}e_j' \right) \prod_{j=m_2+1}^{q_2}\Gamma^{B_j'} \left(1-b_j'-\frac{s_2}{\mu}f_j' \right)}{\prod_{j=1}^{n_2} \Gamma \left(1-a_j'-\frac{s_2}{\mu}e_j' \right) \prod_{j=1}^{m_2} \Gamma\left(b_j'+\frac{s_2}{\mu}f_j' \right)}
\end{align}
then the p.d.f. of the variate $Y= X_1 \; X_2$ is given by 
\begin{align} \label{productrv1}
g(Y) = & \frac{\sigma \;\mu} {y} \; \phi_1(s_1)\; \phi_2(s_2) \nonumber \\
& \times \hat{I}_{\;p_1+p_2,\; q_1+q_2}^{\;m_1+m_2,\; n_1+n_2} \left[ z_1^\mu \; z_2^\sigma \; y^{\mu \sigma}  \left |\begin{array}{c} _1(a_j +\frac{s_1}{\sigma} e_j, \mu e_j, 1)_{n_1},\; _1(a_j' +\frac{s_2}{\mu} e_j', \sigma e_j', 1)_{n_2},\\ _1(b_j+\frac{s_1}{\sigma}f_j, \mu f_j, 1)_{m_1}, \; _1(b_j'+\frac{s_2}{\mu}f_j', \sigma f_j', 1)_{m_2}, \end{array} \right. \right. \nonumber \\
& \left. \left. \begin{array}{c} _{n_1+1}(a_j+\frac{s_1}{\sigma} e_j, \mu e_j ,A_j)_{p_1}, \;_{n_2+1}(a_j'+\frac{s_2}{\mu}e_j'1, \sigma e_j' ,A_j')_{p_2}\\  _{m_1+1}(b_j+\frac{s_1}{\sigma}f_j, \mu f_j, B_j)_{q_1}, \; _{m_2+1}(b_j'+\frac{s_2}{\mu}f_j', \sigma f_j', B_j')_{q_2} \end{array} \right. \right]
\end{align} 
\end{theorem}
\begin{proof}
It is not difficult to see that 
\begin{align}\label{ex1}
E(X_1^{s-1}) &=  \int_0^ \infty x_1^{s-1} \; f_1(X_1) \;  dx_1 \nonumber \\
&= z_1^{\frac{1-s}{\sigma}}\; \phi_1(s_1) \; \theta_1 \left(\frac{s_1-1+s}{\sigma}\right)
\end{align}
and
\begin{align} \label{ex2}
E(X_2^{s-1}) & = \int_0^ \infty x_2^{s-1} \; f_2 (X_2) \;  dx_2 \nonumber \\
&= z_2^{\frac{1-s}{\mu}}\; \phi_2(s_2) \; \theta_2 \left(\frac{s_2-1+s}{\mu}\right)
\end{align}
Now, since $X_1, X_2$ are independent random variables, so
\begin{align}
E(Y^{s-1})= E[(X_1 \; X_2) ^{s-1}] = E(X_1^{s-1}) \; E(X_2^{s-1})
\end{align}
and 
\begin{align}\label{gofy}
g(Y)=  \frac{1}{2\pi i}  \int_{\mathcal{L}} E(Y^{s-1}) \;y^{-s} \;   ds 
\end{align}
Using \eqref{ex1} and \eqref{ex2}  in \eqref{gofy},  we have 
\begin{align}
g(Y)& = \frac{1}{2\pi i}  \int_{\mathcal{L}}z_1^{\frac{1-s}{\sigma}}\; \phi_1(s_1) \; \theta_1 \left(\frac{s_1-1+s}{\sigma}\right)\;   z_2^{\frac{1-s}{\mu}}\; \phi_2(s_2) \; \theta_2 \left(\frac{s_2-1+s}{\mu}\right)\;ds \nonumber \\
& =z_1^\frac{1}{\sigma} \; z_2^\frac{1}{\mu} \; \phi_1(s_1) \; \phi_2(s_2)\nonumber \\
& ~~~ ~~ \times \frac{1}{2 \pi i}\int_{\mathcal{L}} \theta_1 \left(\frac{s_1-1+s}{\sigma}\right)\; \theta_2 \left(\frac{s_2-1+s}{\mu}\right)\; \left(y\;  z_1^\frac{1}{\sigma}\; z_2^\frac{1}{\mu}\right)^{-s} d s 
\end{align}
By using the definition of $\hat{I}$-function \eqref{hatI}, we get
\begin{align}
g(Y)& =z_1^\frac{1}{\sigma} \; z_2^\frac{1}{\mu} \; \phi_1(s_1) \; \phi_2(s_2) \nonumber \\
& \times \hat{I}_{\;p_1+p_2,\; q_1+q_2}^{\;m_1+m_2,\; n_1+n_2} \left[ y\;  z_1^\frac{1}{\sigma}\; z_2^\frac{1}{\mu}  \left |\begin{array}{c} _1(a_j +\frac{s_1-1}{\sigma} e_j, \frac{e_j}{\sigma} , 1)_{n_1},\;_1(a_j' +\frac{s_2-1}{\mu} e_j', \frac{e_j'}{\mu}, 1)_{n_2},\\ _1(b_j+\frac{s_1-1}{\sigma}f_j, \frac{f_j}{\sigma}, 1)_{m_1},\; 1(b_j'+\frac{s_2-1}{\mu}f_j', \frac{f_j'}{\mu}, 1)_{m_2}, \end{array} \right. \right. \nonumber \\
& \qquad \qquad \qquad \left. \left. \begin{array}{c} _{n_1+1}(a_j+\frac{s_1-1}{\sigma} e_j, \frac{e_j}{\sigma} ,A_j)_{p_1}, \;_{n_2+1}(a_j'+\frac{s_2-1}{\mu}e_j', \frac{e_j'}{\mu} ,A_j')_{p_2}\\  \; _{m_1+1}(b_j+\frac{s_1-1}{\sigma}f_j, \frac{f_j}{\sigma}, B_j)_{q_1}, \; _{m_2+1}(b_j'+\frac{s_2-1}{\mu}f_j', \frac{ f_j'}{\mu}, B_j')_{q_2} \end{array} \right. \right]\nonumber
\end{align} 
Upon using \eqref{prop1} and \eqref{prop2}, we finally arrive at \eqref{productrv1}.\\
This completes the proof of Theorem 1. \end{proof}
\begin{theorem} If $X_1$ and $X_2$ be two independent $\hat{I}$-function variates with p.d.f. given respectively by the equations \eqref{f1x1} and \eqref{f2x2}, where $0<\sigma <\infty, 0<\mu <\infty, z_1, z_2, s_1, s_2 >0$, and $\phi_1(s_1)$ and $\phi_2(s_2)$ are also the same given by equations \eqref{phi1s1} and \eqref{phi2s2} respectively, then the p.d.f. of the variate $Y= \frac{X_1}{X_2}$ is given by 
\begin{align} \label{quotientrv1}
g(Y) = & \frac{\sigma \;\mu} {y} \; \phi_1(s_1)\; \phi_2(s_2) \nonumber \\
& \times \hat{I}_{\;p_1+q_2,\; q_1+p_2}^{\;m_1+n_2,\; n_1+m_2} \left[ z_1^\mu \; z_2^{-\sigma} \; y^{\mu \sigma}  \left |\begin{array}{c} _1(a_j +\frac{s_1}{\sigma} e_j, \mu e_j, 1)_{n_1},\\ _1(b_j+\frac{s_1}{\sigma}f_j, \mu f_j, 1)_{m_1}, \end{array} \right. \right. \nonumber \\
& \qquad \qquad \qquad \qquad \left. \left. \begin{array}{c} _1(1-b_j' -\frac{s_2}{\mu} f_j', \sigma f_j', 1)_{m_2},\;_{n_1+1}(a_j+\frac{s_1}{\sigma} e_j, \mu e_j ,A_j)_{p_1}, \\ _1(1-a_j'-\frac{s_2}{\mu}e_j', \sigma e_j', 1)_{n_2}, \; _{m_1+1}(b_j+\frac{s_1}{\sigma}f_j, \mu f_j, B_j)_{q_1},  \end{array} \right. \right. \nonumber \\
& \qquad \qquad \qquad \qquad \qquad \qquad \left. \left. \begin{array}{c} _{m_2+1}(1-b_j'-\frac{s_2}{\mu}f_j', \sigma f_j'1 ,B_j')_{q_2}\\ _{n_2+1}(1-a_j'-\frac{s_2}{\mu}e_j', \sigma e_j', A_j')_{p_2} \end{array} \right. \right]
\end{align} 
\end{theorem}
\begin{proof}
We have
\begin{align}
E(Y^{s-1})= E\left[ \left (\frac{X_1} {X_2} \right )^{s-1} \right] & =  \frac {E(X_1^{s-1})} {E(X_2^{s-1})} =E(X_1^{s-1}X_2^{1-s})\nonumber
\end{align}
Since $X_1$ and $X_2$ are independent random variables, 
\begin{align}  
 E(Y^{s-1})= E(X_1^{s-1}) E(X_2^{1-s})\nonumber
\end{align}
Using \eqref{ex1} and \eqref{ex2} in \eqref{gofy}, we have 
\begin{align}
g(Y)& = \frac{1}{2\pi i}  \int_{\mathcal{L}}z_1^{\frac{1-s}{\sigma}}\; \phi_1(s_1) \; \theta_1 \left(\frac{s_1-1+s}{\sigma}\right)\;   z_2^{\frac{s-1}{\mu}}\; \phi_2(s_2) \; \theta_2 \left(\frac{s_2+1-s}{\mu}\right)\;ds \nonumber \\
& =z_1^\frac{1}{\sigma} \; z_2^{-\frac{1}{\mu}} \; \phi_1(s_1) \; \phi_2(s_2)\nonumber \\
& ~~~ ~~ \times \frac{1}{2 \pi i}\int_{\mathcal{L}} \theta_1 \left(\frac{s_1-1+s}{\sigma}\right)\; \theta_2 \left(\frac{s_2+1-s}{\mu}\right)\; \left(y\;  z_1^\frac{1}{\sigma}\; z_2^{-\frac{1}{\mu}}\right)^{-s} d s 
\end{align}
By using the definition of $\hat{I}$-function \eqref{hatI}, we get
\begin{align}
g(Y)& =z_1^\frac{1}{\sigma} \; z_2^{-\frac{1}{\mu}} \; \phi_1(s_1) \; \phi_2(s_2) \nonumber \\
& \times \hat{I}_{\;p_1+q_2,\; q_1+p_2}^{\;m_1+n_2,\; n_1+m_2} \left[ y\;  z_1^\frac{1}{\sigma}\; z_2^{-\frac{1}{\mu}}  \left |\begin{array}{c} _1(a_j +\frac{s_1-1}{\sigma} e_j, \frac{e_j}{\sigma} , 1)_{n_1},\\ _1(b_j+\frac{s_1-1}{\sigma}f_j, \frac{f_j}{\sigma}, 1)_{m_1}, \end{array} \right. \right. \nonumber \\
& \qquad \qquad\left. \left. \begin{array}{c} _1(1-b_j' -\frac{s_2+1}{\mu} f_j', \frac{f_j'}{\mu}, 1)_{m_2},\;_{n_1+1}(a_j+\frac{s_1-1}{\sigma} e_j, \frac{e_j} {\sigma} ,A_j)_{p_1}, \\ _1(1-a_j'-\frac{s_2+1}{\mu}e_j', \frac{e_j'}{\mu}, 1)_{n_2}, \; _{m_1+1}(b_j+\frac{s_1-1}{\sigma}f_j, \frac{f_j}{\sigma}, B_j)_{q_1},   \end{array} \right. \right. \nonumber \\
& \qquad\qquad \qquad \qquad\left. \left. \begin{array}{c} _{m_2+1}(1-b_j'-\frac{s_2+1}{\mu}f_j', \frac {f_j'}{\mu} ,B_j')_{q_2}\\ _{n_2+1}(1-a_j'-\frac{s_2+1}{\mu}e_j', \frac{ e_j'}{\mu}, A_j')_{p_2} \end{array} \right. \right]
\end{align} 
Upon using \eqref{prop1} and \eqref{prop2}, we finally arrive at \eqref{quotientrv1}. \\
This completes the proof of Theorem 2.
\end{proof}
\section{Corollaries}
\subsection*{Corollary 1}
If $A_j=1$ for $j=n_1+1, \cdots,p_1$, $B_j=1$ for $j=m_1+1, \cdots,q_1$, $A_j'=1$ for $j=n_2+1, \cdots,p_2$, and $B_j'=1$ for $j=m_2+1, \cdots,q_2$, then the $\hat{I}$-function in Thereom 1 reduces to H-function, and we get the corresponding result for the H-function as follows. \\
 If $X_1$ and $X_2$ are independent $H$-function variates with p.d.f. given respectively by 
\begin{align} 
f_1(X_1) = \sigma \; z_1^{\frac{s_1}{\sigma}} \; \phi_1(s_1) \; x^{s_1-1}H_{\;p_1,\; q_1}^{\;m_1,\; n_1} \left[ z_1\; x^\sigma \left |\begin{array}{c} _1(a_j, e_j)_{p_1}\\ _1(b_j, f_j)_{q_1}\end{array} \right] \right., x>0
\end{align}
and 
\begin{align} 
f_2(X_2) = \mu \; z_2^{\frac{s_2}{\mu}} \; \phi_2(s_2) \; x^{s_2-1}H_{\;p_2,\; q_2}^{\;m_2,\; n_2} \left[ z_2\; x^\mu \left |\begin{array}{c} _1(a_j', e_j')_{p_2}\\ _1(b_j', f_j')_{q_2}\end{array} \right] \right., x>0
\end{align}
where $0<\sigma <\infty, 0<\mu <\infty, z_1, z_2, s_1, s_2 >0$, \\
also
\begin{align} 
\phi_1(s_1) = \frac{\prod_{j=n_1+1}^{p_1} \Gamma \left(a_j+\frac{s_1}{\sigma}e_j \right) \prod_{j=m_1+1}^{q_1}\Gamma \left(1-b_j-\frac{s_1}{\sigma}f_j \right)}{\prod_{j=1}^{n_1} \Gamma \left(1-a_j-\frac{s_1}{\sigma}e_j \right) \prod_{j=1}^{m_1} \Gamma\left(b_j+\frac{s_1}{\sigma}f_j \right)}
\end{align} 
and
\begin{align} 
\phi_2(s_2) = \frac{\prod_{j=n_2+1}^{p_2} \Gamma \left(a_j'+\frac{s_2}{\mu}e_j' \right) \prod_{j=m_2+1}^{q_2}\Gamma \left(1-b_j'-\frac{s_2}{\mu}f_j' \right)}{\prod_{j=1}^{n_2} \Gamma \left(1-a_j'-\frac{s_2}{\mu}e_j' \right) \prod_{j=1}^{m_2} \Gamma\left(b_j'+\frac{s_2}{\mu}f_j' \right)}
\end{align}
then the p.d.f. of the variate $Y= X_1 \; X_2$ is given by 
\begin{align} \label{productrv1H}
g(Y) &=  \frac{\sigma \;\mu} {y} \; \phi_1(s_1)\; \phi_2(s_2) \nonumber \\
& \times H_{\;p_1+p_2,\; q_1+q_2}^{\;m_1+m_2,\; n_1+n_2} \left[ z_1^\mu \; z_2^\sigma \; y^{\mu \sigma}  \left |\begin{array}{c} _1(a_j +\frac{s_1}{\sigma} e_j, \mu e_j)_{n_1},\\ _1(b_j+\frac{s_1}{\sigma}f_j, \mu f_j)_{m_1}, \end{array} \right. \right. \nonumber \\
& \left. \left. \begin{array}{c} _1(a_j' +\frac{s_2}{\mu} e_j', \sigma e_j')_{n_2},\;_{n_1+1}(a_j+\frac{s_1}{\sigma} e_j, \mu e_j)_{p_1}, \;_{n_2+1}(a_j'+\frac{s_2}{\mu}e_j', \sigma e_j')_{p_2}\\ _1(b_j'+\frac{s_2}{\mu}f_j', \sigma f_j')_{m_2}, \; _{m_1+1}(b_j+\frac{s_1}{\sigma}f_j, \mu f_j)_{q_1}, \; _{m_2+1}(b_j'+\frac{s_2}{\mu}f_j', \sigma f_j')_{q_2} \end{array} \right. \right]
\end{align} 
 The p.d.f. of the variate $Y= \frac{X_1}{X_2}$ is given by 
\begin{align} \label{quotientrv1H}
h(Y) &=  \frac{\sigma \;\mu} {y} \; \phi_1(s_1)\; \phi_2(s_2) \nonumber \\
& \times H_{\;p_1+q_2,\; q_1+p_2}^{\;m_1+n_2,\; n_1+m_2} \left[ z_1^\mu \; z_2^{-\sigma} \; y^{\mu \sigma}  \left |\begin{array}{c} _1(a_j +\frac{s_1}{\sigma} e_j, \mu e_j)_{n_1},\\ _1(b_j+\frac{s_1}{\sigma}f_j, \mu f_j)_{m_1}, \end{array} \right. \right. \nonumber \\
& \left. \left. \begin{array}{c} _1(1-b_j' -\frac{s_2}{\mu} f_j', \sigma f_j')_{q_2},\;_{n_1+1}(a_j+\frac{s_1}{\sigma} e_j, \mu e_j)_{p_1}\\ _1(1-a_j'-\frac{s_2}{\mu}e_j', \sigma e_j')_{p_2}, \; _{m_1+1}(b_j+\frac{s_1}{\sigma}f_j, \mu f_j)_{q_1} \end{array} \right. \right]
\end{align} 
\subsection*{Corollary 2}
If we take $\sigma = \mu=r, z_1=z_2=z, m_1=m_2=m, n_1=n_2=n, p_1=p_2=p, q_1=q_2=q$ in Theorem 2, then we get the following result after little simplification.\\ 
If $X_1$ and $X_2$ are independent $\hat{I}$-function variates with p.d.f. given respectively by 
\begin{align} 
f_1(X_1) = r \; z^{\frac{s_1}{r}} \; \phi_1(s_1) \; x^{s_1-1}\hat{I}_{\;p,\; q}^{\;m,\; n} \left[ z\; x^r \left |\begin{array}{c} _1(a_j, e_j, 1)_n,\; _{n+1}(a_j, e_j,A_j)_p\\ _1(b_j, f_j, 1)_m,\; _{m+1}(b_j, f_j, B_j)_q\end{array} \right] \right.
\end{align}
and 
\begin{align} 
f_2(X_2) =  r \; z^{\frac{s_2}{r}} \; \phi_2(s_2) \; x^{s_2-1}\hat{I}_{\;p,\; q}^{\;m,\; n} \left[ z\; x^r \left |\begin{array}{c} _1(a_j, e_j, 1)_n,\; _{n+1}(a_j, e_j,A_j)_p\\ _1(b_j, f_j, 1)_m,\; _{m+1}(b_j, f_j, B_j)_q\end{array} \right] \right.
\end{align}
where $0<r <\infty, z, s_1, s_2 >0$, \\
 The p.d.f. of the variate $Y= \frac{X_1}{X_2}$ is given by 
\begin{align} \label{quotientrvcor2}
g(Y) = & r\; y^{-s_2-1} \; \phi(s_1)\; \phi(s_2) \nonumber \\
& \times \hat{I}_{\;p+q,\; q+p}^{\;m+n,\; n+m} \left[ y^r  \left |\begin{array}{c} _1(1-b_j,f_j, 1)_m,\; _{m+1}(1-b_j,f_j,B_j)_q, \\ _1(b_j+\frac{s_1+s_2}{r}f_j, f_j, 1)_m,\; _{m+1}(b_j+\frac{s_1+s_2}{r}f_j, f_j, B_j)_q,  \end{array} \right. \right. \nonumber \\
& \left. \left. \begin{array}{c} _1(a_j+\frac{s_1+s_2}{r} e_j, e_j, 1)_n,\;_{n+1}(a_j+\frac{s_1+s_2}{r} e_j, e_j ,A_j)_p \\_1(1-a_j, e_j,1)_n, \;  _{n+1}(1-a_j, e_j, A_j)_p \end{array} \right. \right]
\end{align} 
Further, in \eqref{quotientrvcor2}, if we take $A_j=1$ for $j=n+1, \cdots,p$ and $B_j=1$ for $j=m+1, \cdots,q$, then we get a known theorem due to Mathai and Saxena \cite{mathaisaxena}.\\
Similarly, other special cases can be considered. 

\section*{Concluding Remarks }
In this paper, in the beginning, we have obtained the Mellin transforms of the $ \hat{I}$ - function and Mellin transforms of the product of two $\hat{I}$ - functions. Thereafter, a general probability distribution involving $\hat{I}$-function has been introduced. Finally, we have established the distributions of the product and quotient of two independent random variables having densities in terms of the $\hat{I}$ - functions (which is a particular case of the Rathie's ${I}$ - function). Since $\hat{I}$- function is the most general function of one variable which includes not only the well known functions such as Inayat Hussain's $\bar{H}$-function\cite{inayat}, Fox's H-function\cite{fox, mathai} and Meijer's G-function\cite{mathai}, but also a large number of elementary functions, so on specializing the parameters therein, we can obtain distributions  of product and quotient of two independent random variables having different densities. It is hoped that the results presented in this paper could be of potential use in the area of statistics, applied mathematics,  mathematical physics and wireless communication. \\

{\it We conclude this paper by remarking that all the results established in the paper contain proper convergence   conditions.}

\end{document}